\documentclass[11 pt]{amsart}

\usepackage[margin=1.3in]{geometry}

\usepackage{amssymb}
\usepackage{amsmath}
\usepackage{amsthm}
\usepackage{amsfonts}
\usepackage{mathrsfs}
\usepackage{amsxtra}
\usepackage{stmaryrd}
\usepackage[all,2cell]{xy}
\usepackage{verbatim}
\usepackage{color}
\usepackage[unicode=true, pdfusetitle,
 bookmarks=true,bookmarksnumbered=false,
 breaklinks=false,
 backref=false,
 colorlinks=true,
 linkcolor=blue,
 citecolor=blue,
 urlcolor=blue,
 final
]{hyperref}
\usepackage[
textwidth=3cm,
textsize=small,
colorinlistoftodos]
{todonotes}

\usepackage{bookmark}
\usepackage{graphicx}

\usepackage[capitalise]{cleveref}
\usepackage{amssymb,amsmath,tikz-cd,dsfont}
\usepackage{tikz}
\usepackage{tikz-qtree}
\usepackage{comment}
\usetikzlibrary{arrows,matrix,positioning}

\newcommand{\newrefformat}[2]{}

\theoremstyle{plain}   
\newtheorem{thm}{Theorem}[section] 
\newtheorem{cor}{Corollary}[section]
\makeatletter\let\c@cor\c@thm\makeatother
\newtheorem{lemma}{Lemma}[section]
\makeatletter\let\c@lemma\c@thm\makeatother
\newtheorem{prop}{Proposition}[section]
\makeatletter\let\c@prop\c@thm\makeatother

\makeatletter\let\c@claim\c@thm\makeatother

\theoremstyle{definition}

\newtheorem{defn}{Definition}[section]
\makeatletter\let\c@defn\c@thm\makeatother
\newtheorem{const}{Construction}[section]
\makeatletter\let\c@const\c@thm\makeatother

\makeatletter\let\c@notn\c@thm\makeatother

\makeatletter\let\c@outline\c@thm\makeatother

\theoremstyle{remark}

\newtheorem{rem}{Remark}[section]
\makeatletter\let\c@rem\c@thm\makeatother
\newtheorem{ex}{Example}[section]
\makeatletter\let\c@ex\c@thm\makeatother

\makeatletter\let\c@observationn\c@thm\makeatother

\makeatletter
\let\c@equation\c@thm
\numberwithin{equation}{section}
\makeatother

\crefname{lemma}{Lemma}{Lemmas}
\crefname{thm}{Theorem}{Theorems}
\crefname{defn}{Definition}{Definitions}
\crefname{notn}{Notation}{Notations}
\crefname{const}{Construction}{Constructions}
\crefname{prop}{Proposition}{Propositions}
\crefname{rem}{Remark}{Remarks}
\crefname{cor}{Corollary}{Corollaries}
\crefname{equation}{Display}{Displays}
\crefname{ex}{Example}{Examples}

\usepackage{tikzit}

\tikzstyle{Vertex}=[fill=black, draw=white, shape=circle, tikzit shape=circle, scale=0.4]
\tikzstyle{red rectangle}=[fill=red, draw=black, shape=rectangle]
\tikzstyle{smallDot}=[fill=black, draw=black, shape=circle, scale=.1]

\tikzstyle{new edge style 0}=[-, draw=red]
\tikzstyle{new edge style 1}=[-, draw={rgb,255: red,1; green,69; blue,255}]
\tikzstyle{new edge style 2}=[-, draw={rgb,255: red,0; green,255; blue,42}]
\tikzstyle{dotted edge}=[-, dashed]
\tikzstyle{new edge style 3}=[{|->}]
\tikzstyle{new edge style 4}=[{|->}, dashed]
\tikzstyle{purple line}=[-, draw={rgb,255: red,150; green,32; blue,252}]
\tikzstyle{forest line}=[-, draw={rgb,255: red,37; green,112; blue,37}]
\tikzstyle{Red arrow}=[draw=red, {|->}]
\tikzstyle{Orange Edge}=[-, draw={rgb,255: red,255; green,128; blue,0}]
\tikzstyle{Dotted Green Arrow}=[draw=green, {|->}, dashed]
\tikzstyle{Dotted Red Line}=[-, draw=red, dashed]
\tikzstyle{new edge style 5}=[->]

\usepackage{stmaryrd}
\usepackage{tikz}

\usepackage{tikz-cd}
\usepackage{float}

\usepackage{comment}
\usetikzlibrary{arrows,matrix,positioning}
\usepackage[normalem]{ulem}

\usepackage{tikz}

\usepackage{tikz-cd}
\usepackage{float}

\usetikzlibrary{calc}
\usetikzlibrary{decorations.markings,decorations.pathmorphing}
\tikzset{/tikz/commutative diagrams/arrow style=tikz,>=stealth} 
\tikzset{mm/.style={execute at begin node=$\displaystyle, execute at end node=$}}
\tikzset{tikzob/.style={commutative diagrams/every diagram, every cell}}
\tikzset{tikzar/.style={commutative diagrams/.cd, every arrow, every label, font={\small}}}
\tikzset{tikzsquiggle/.style={decorate, decoration={
    snake,
    segment length=8pt,
    amplitude=.9pt,post=lineto,
    post length=2pt}}}
\tikzset{cross line/.style={preaction={draw=white, -, line width=6pt}}}

\newcommand{\NC}{\operatorname{NC}}

\newcommand{\Hom}{{\rm{Hom}}}

\newcommand{\F}{\mathcal{F}}

\newcommand{\Co} {\mathbb{C}}

\newcommand \transfers {\operatorname{Tr}}

\newcommand \DE {\mathcal{E}}

\newcommand{\cB}{\mathcal{B}}
\newcommand{\cC}{\mathcal{C}}
\newcommand{\cD}{\mathcal{D}}

\newcommand{\cL}{\mathcal{L}}
\newcommand{\cM}{\mathcal{M}}
\newcommand{\cN}{\mathcal{N}}
\newcommand{\cO}{\mathscr{O}}
\newcommand{\cP}{\mathcal{P}}

\newcommand{\cR}{\mathcal{R}}

\newcommand{\Ho}{\operatorname{Ho}}
\newcommand{\NOp}{N_\infty\text{-}\mathbf{Op}}

\newcommand \sub{\mathrm{Sub}}
\newcommand \trans{\mathrm{Tr}}
\newcommand \op{\mathrm{op}}
\newcommand{\Cat}{\operatorname{Cat}}
\newcommand{\rR}{\;\mathcal{R}\;}

\title{Self-duality of the lattice of transfer systems \\ via weak factorization systems}

\author{Evan E. Franchere}
\address{Department of Mathematics, Reed College, Portland, OR 97202, USA}
\email{franchev@reed.edu}

\author{Kyle Ormsby}
\address{Department of Mathematics, Reed College, Portland, OR 97202, USA}
\email{ormsbyk@reed.edu}

\author{Ang\'{e}lica M. Osorno}
\address{Department of Mathematics, Reed College, Portland, OR 97202, USA}
\email{aosorno@reed.edu}

\author{Weihang Qin}
\address{Department of Mathematics, Reed College, Portland, OR 97202, USA}
\email{qinw@reed.edu}

\author{Riley Waugh}
\address{Department of Mathematics, Reed College, Portland, OR 97202, USA}
\email{waughr@reed.edu}

\begin{document}
\begin{abstract}
For a finite group $G$, $G$-transfer systems are combinatorial objects which encode the homotopy category of $G$-$N_\infty$ operads, whose algebras in $G$-spectra are $E_\infty$ $G$-spectra with a specified collection of multiplicative norms. For $G$ finite Abelian, we demonstrate a correspondence between $G$-transfer systems and weak factorization systems on the poset category of subgroups of $G$. This induces a self-duality on the lattice of $G$-transfer systems.
\end{abstract}

\maketitle

\section{Introduction}\label{section:intro}

In stable homotopy theory, commutative ring spectra are algebras over an $E_\infty$ operad.  The equivariant story is more subtle.  In \cite{HillBlumberg}, A.~Blumberg and M.~Hill introduced $N_\infty$ operads to capture the varying classes of multiplicative norm maps supported by equivariant ring spectra. By work of Blumberg and Hill \cite{HillBlumberg}, P.~Bonventre and L.~Pereira \cite{BP}, J.~Guti\'errez and D.~White \cite{GW}, and J.~Rubin \cite{rubin_Ninfty}, we know that the homotopy category of $N_\infty$ operads may be identified with the lattice of indexing systems for the group of equivariance. The indexing system associated to an operad records which norm maps are being encoded by the operad. Further work of Rubin \cite{rubin_steiner} and  S.~Balchin, D.~Barnes, and C.~Roitzheim \cite{CPCatalan}, identified transfer systems as the essential combinatorial data of indexing systems, thus proving that the homotopy category of $N_\infty$ operads is equivalent to the lattice of transfer systems. The combinatorics of this lattice thus plays a central role in the study of equivariant ring spectra.

Fix a finite group $G$. We recall the basics of $G$-$N_\infty$ operads and $G$-transfer systems in \cref{secn:trans}.  For the purposes of this introduction, note that a $G$-transfer system is a relation $\cR$ on $\sub(G)$, the subgroup lattice of $G$, that refines inclusion\footnote{A relation $\rR$ refines the inclusion relation $\le$ when $H \rR K$ implies $H\le K$.} and satisfies the following conditions:
\begin{itemize}
    \item (reflexivity) $H \rR H$ for all $H\leq G$,
    \item (transitivity) $K \rR H$ and $L \rR K$ implies $L \rR H$,
    \item (closed under conjugation) $K \rR H$ implies that $(gKg^{-1}) \rR (gHg^{-1})$ for all $g\in G$,
    \item (closed under restriction) $K \rR H$ and $M\leq H$ implies $(K \cap M) \rR M$.
\end{itemize}
In other words, a transfer system is a sub-poset of $\sub(G)$ closed under conjugation and restriction.  These objects form a bounded lattice $\trans(G)$ ordered under refinement (see \cref{prop:bounded lattice}).

In this paper, we establish that $\trans(G)$ is self-dual whenever $G$ is Abelian. Our proof proceeds via a surprising connection with weak factorization systems.  A weak factorization system on a category $\cC$ is a pair of classes of morphisms $(\cL,\cR)$ satisfying the factorization and lifting axioms specified in \cref{defn: WFS}.  Considering $\sub(G)$ as a poset category, we show in \cref{thm: 1-1 WFS and TS} that there is a bijective correspondence between transfer systems and weak factorization systems given by
\[\cR \longleftrightarrow ({}^\boxslash \cR,\cR)\]
where ${}^\boxslash \cR$ denotes the morphisms in $\sub(G)$ having the left-lifting property with respect to $\cR$.

When $G$ is Abelian, the subgroup lattice $\sub(G)$ carries a self-duality $\nabla$. In \cref{thm:cat-involution}, we prove that
\[
\begin{aligned}
  \phi\colon \trans(G)&\longrightarrow \trans(G)\\
  \cR&\longmapsto (({}^\boxslash \cR)^\op)^\nabla
\end{aligned}{}
\]
is a self-duality on $\trans(G)$. This generalizes the self-duality on transfer systems for cyclic groups of squarefree order observed in \cite{BBPR}; see \cref{thm:compare BBPR}. We anticipate that the self-duality of $\trans(G)$ will prove useful in future enumerative work on transfer systems and $N_\infty$ operads.

\subsection*{Organization}
\cref{secn:posets,secn:trans} cover necessary background material on partially ordered sets and transfer systems, respectively.

The real work is contained in \cref{section: Category Stuff}, which is organized into five subsections.  Subsection \ref{trans sys on poset} abstracts the notion of a transfer system for an Abelian group to the context of arbitrary posets.  Subsection \ref{sec:wfs} introduces weak factorization systems and proves that they are in bijection with transfer systems on a poset (\cref{thm: 1-1 WFS and TS}). In subsection \ref{sec:self-dual}, we prove \cref{thm:cat-involution} on self-duality of the transfer system lattice.  Subsection \ref{sec:BBPR} compares our self-duality (which is defined for every finite Abelian group) to the self-duality of \cite{BBPR} on $\trans(G)$ for $G$ cyclic of squarefree order. In Subsection \ref{sec:slats}, we illustrate a numerical symmetry of the duality, namely the number of ``slats'' in a transfer system on a cyclic group of order $p^nq$ for $p,q$ distinct primes.

Finally, \cref{sec:nc} produces a direct bijection between transfer systems for a cyclic group of order $p^n$, $p$ prime, and noncrossing partitions of $\{0,1,\ldots,n\}$.  This gives a novel proof of the Catalan enumeration of such transfer systems originally found in \cite{CPCatalan}, and we deduce a new corollary linking minimal generation of transfer systems to the Narayana numbers. This section is independent of the rest of the paper.

\subsection*{Notation}
We use the following notation throughout. 
\begin{itemize}
  \item $G$ --- a finite group, eventually Abelian.
  \item $\sub(G)$ --- the subgroup lattice of $G$.
  \item $\trans(G)$ --- the lattice of transfer systems on $G$ under refinement.
  \item $\cP$ --- a poset, considered either as a set with a relation or as a category in which there is at most one morphism between each pair of objects.
  \item $[n]$ --- the poset $\{0<1<\cdots<n\}$.
  \item $\cB_n$ --- the Boolean poset of subsets of $\{1,2,\ldots,n\}$ under inclusion.
  \item $\cD_n$ --- the divisor poset of $n$ under divisibility.
  \item $C_n$ --- the cyclic group of order $n$.
  \item $\cP^\op$ --- the dual of a poset $\cP$.
  \item $\cR$ --- a transfer system, considered either as a relation or a collection of morphisms.
  \item ${}^\boxslash \cM$ and $\cM^\boxslash$ --- morphisms with the left (resp.~right) lifting property with respect to a class of morphisms $\cM$.
  \item $\cL\boxslash \cR$ --- the property $\cL\subseteq {}^\boxslash \cR$ (equivalently, $\cR\subseteq \cL^\boxslash$).
  \item $\nabla$ --- self-duality on a poset.
\end{itemize}

\subsection*{Acknowledgements}
The authors thank Jonathan Rubin, who suggested this topic, and Constanze Roitzheim, who explained the contents of \cite{CPCatalan, BBPR}, and suggested avenues of exploration. We also thank Hugh Robinson for helpful correspondence and for authoring OEIS entry A092450.\footnote{Computer calculations of transfer systems that agreed with A092450 were our first hint at the link between transfer and weak factorization systems.} Finally, we thank the anonymous referee for helpful suggestions regarding the exposition in this paper. This research was supported by NSF grant DMS-1709302.

\section{Preliminaries on posets}\label{secn:posets}

In this section, we briefly collect some well-known facts and examples from the theory of partially ordered sets and lattices. We refer the reader to \cite[Chapter 3]{Stanley} for a comprehensive reference.

Recall that a \emph{partially ordered set} or \emph{poset} $(\cP,\leq)$ consists of a set $\cP$, together with a binary relation that is reflexive, antisymmetric and transitive. We say that $x<y$ is a \emph{cover relation} in $\cP$ if there is no $z\in \cP$ such that $x<z<y$.

We represent a finite poset $\cP$ with a \emph{Hasse diagram}. This is a directed graph whose vertices are the elements of $\cP$, edges are cover relations, and such that if $x<y$, then $x$ is drawn below $y$. 

\begin{ex}\label{ex:posets} Let $n$ be a natural number. See \cref{fig:hasse} below for the Hasse diagrams of the following posets.
 \begin{enumerate}
  \item\label{linear} We denote by $[n]$ the set $\{0,1,\dots,n\}$ with its usual order structure. 
  \item\label{Boolean} The Boolean poset $\cB_n$ is the poset of subsets of $\{1,2,\dots,n\}$ under inclusion. It is isomorphic to the product of $n$ copies of $[1]$ with itself. The Hasse diagram corresponds to the edges of an $n$-dimensional cube.
  \item\label{divisors} The set of positive divisors of $n$ forms a poset $\cD_n$ under divisibility. If $n=p_1^{a_1}\dots p_k^{a_k}$ is the prime factorization of $n$, then $\cD_n$ is isomorphic to $[a_1]\times \dots \times [a_k]$ by identifying the exponents of the primes. Its Hasse diagram is a $k$-dimensional grid.
  \item\label{subgroups} For a group $G$, we denote by $\sub(G)$ the poset of subgroups of $G$ under inclusion. Note that $\sub(C_n)$ is isomorphic to $\cD_n$, where $C_n$ denotes the cyclic group of order $n$.
 \end{enumerate}
\end{ex}

\begin{center}
\begin{figure}
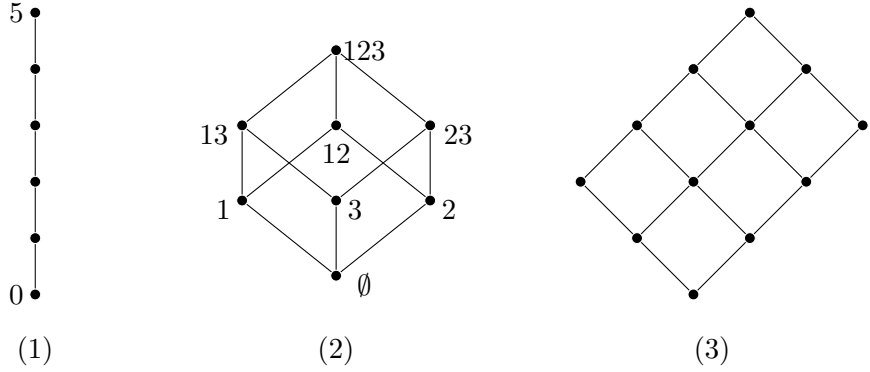

\ctikzfig{HasseDiagrams}
\caption{Hasse diagrams for (1) $[5]$, (2) $\cB_3$, and (3) $[3]\times [2] \cong \cD_{p^3q^2} \cong \sub(C_{p^3q^2})$ for $p,q$ distinct primes.}\label{fig:hasse}
\end{figure}
\end{center}

For $x$ and $y$ in a poset $\cP$, their least upper bound, if it exists, is denoted by $x \vee y$ and is called the \emph{join}. Similarly, their greatest lower bound is denoted by $x\wedge y$ and is called the \emph{meet}. A \emph{lattice} is a poset for which every pair of elements has both a join and a meet. A poset is \emph{bounded} if it has a least and a greatest element.

All of the posets of \cref{ex:posets} are bounded lattices. In the case of $\sub(G)$, the meet of two subgroups is given by their intersection, while the join is the subgroup generated by their union.

\begin{rem}\label{rem:meet-lattice}
As noted in \cite[Proposition 3.3.1]{Stanley}, if a finite poset $\cP$ has all meets and has a greatest element, then it is a lattice. Dually, if $\cP$ has all joins and has a least element, then it is a lattice.
\end{rem}

Given a poset $\cP$, we consider it as a category whose objects are the elements of $\cP$ and whose morphisms are given by the relation $\leq$. In other words, there is a unique morphism from $x$ to $y$ whenever $x\leq y$, and no morphisms otherwise. Note that every diagram in $\cP$ commutes.

The \emph{dual} of $\cP$, denoted by $\cP^\op$, is the poset with the same underlying set but with relation reversed. Note that as categories, $\cP^\op$ is precisely the opposite category of $\cP$.

\section{Transfer systems and $N_\infty$ operads}\label{secn:trans} 

Transfer systems, as originally and independently defined by J.~Rubin \cite[Definition 3.4]{rubin_steiner} and S.~Balchin, D.~Barnes, and C.~Roitzheim \cite{CPCatalan}, are meant to isolate the essential data necessary to record all the norms/transfer maps encoded by $N_\infty$ operads. Here we recall the necessary definitions, make the preceding statement precise, and record some basic facts about transfer systems.

Let $G$ be a finite group and for $n\ge 0$ let $\mathfrak S_n$ denote the symmetric group on $n$ letters.

\begin{defn}
A \emph{$G$-$N_\infty$ operad} is a symmetric operad $\cO$ on $G$-spaces satisfying the following three properties:
\begin{itemize}
\item for all $n\ge 0$, $\cO(n)$ the $G\times \mathfrak S_n$-space is $\mathfrak S_n$-free,
\item for every $\Gamma\le G\times \mathfrak S_n$, the $\Gamma$-fixed point space $\cO(n)^\Gamma$ is empty or contractible, and
\item for all $n\ge 0$, $\cO(n)^G$ is nonempty.
\end{itemize}
A \emph{map of $G$-$N_\infty$ operads} $\varphi\colon \cO_1\to\cO_2$ is a morphism of operads in $G$-spaces, and as such, the map at level $n$ is $G\times \mathfrak S_n$-equivariant.  The associated category of $G$-$N_\infty$ operads is denoted $\NOp^G$.

A map $\varphi\colon \cO_1\to\cO_2$ of $G$-$N_\infty$ operads is a \emph{weak equivalence} if $\varphi\colon \cO_1(n)^\Gamma\to \cO_2(n)^\Gamma$ is a weak homotopy equivalence of topological spaces for all $n\ge 0$ and $\Gamma\le G\times \mathfrak S_n$. The associated homotopy category (formed by inverting weak equivalences) is denoted $\Ho(\NOp^G)$.
\end{defn}

\begin{rem}
Every $G$-$N_\infty$ operad $\cO$ is a naive $E_\infty$ operad, and thus parametrizes an operation that is associative and commutative up to coherent higher homotopies. In addition, $\cO$ \emph{admits a $T$-norm} for those finite $H$-sets $T$ for which $\cO(|T|)^{\Gamma(T)}$ is nonempty, where $H$ is a subgroup of $G$ and $\Gamma(T)\le G\times \mathfrak S_{|T|}$ is the graph of a permutation representation of $T$, and moreover, these norms are compatible up to coherent homotopies. For a $G$-space $X$, a $T$-norm is a $G$-map
\[G\times_H X^T \longrightarrow X,\]
where $X^T$ denotes the $H$-space of all maps from $T$ to $X$, with $H$ acting by conjugation.
In particular, if $K\leq H$ and $T=H/K$, a $T$-norm induces a ``wrong-way map''
\[X^K \longrightarrow X^H\]
between fixed-point spaces (see \cite[\S6]{HillBlumberg} and \cite[Remark 3.5]{rubin_steiner} for further discussion).
This is the sense in which $G$-$N_\infty$ operads parametrize admissible norms.
\end{rem}

We now turn to transfer systems, which we will eventually relate back to $N_\infty$ operads. Recall that $\sub(G)$ denotes the poset of subgroups of $G$ under inclusion. For $g\in G$ and $H\in \sub(G)$, let ${}^gH = gHg^{-1}\in \sub(G)$ denote the $g$-conjugate of $H$.

\begin{defn}
Let $G$ be a finite group. A \emph{$G$-transfer system} is a relation $\cR$  on $\sub(G)$ that refines the inclusion relation and satisfies the following properties:
\begin{itemize}
    \item (reflexivity) $H \rR H$ for all $H\leq G$,
    \item (transitivity) $K \rR H$ and $L \rR K$ implies $L \rR H$,
    \item (closed under conjugation) $K \rR H$ implies that ${}^gK \rR {}^gH$ for all $g\in G$,
    \item (closed under restriction) $K \rR H$ and $M\leq H$ implies $(K \cap M) \rR M$.
\end{itemize}
\end{defn}

\begin{center}
\begin{figure}
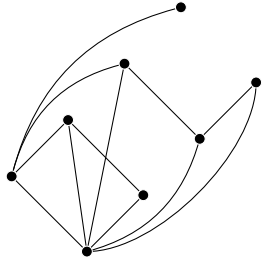

    \ctikzfig{transfexample}
    \caption{An example of a transfer system on $C_{p^3q}$.}\label{fig:transfexample}
\end{figure}
\end{center}

A $G$-transfer system $\cR$ can alternatively be described as a partial order on $\sub(G)$ that refines $\leq$ and is closed under conjugation and under restriction.

We represent transfer systems by drawing the corresponding directed graph, ignoring the trivial edges (\emph{i.e.}, self loops), see \cref{fig:transfexample}. Note that this is not the Hasse diagram for the corresponding poset, as it will include non-covering relations.

\begin{rem}
If $G$ is a Dedekind group (so all subgroups are normal) conjugation is trivially satisfied. We will later concentrate on Abelian groups, the most common class of Dedekind groups.
\end{rem}

\begin{defn}
Let $\trans(G)$ denote the poset of all $G$-transfer systems ordered under refinement. Thus, $\cR_1 \leq \cR_2 $ if and only if  for all $ K,H \in \sub(G)$, if $K\;\cR_1\; H$ then $K\; \cR_2 \;H$.
\end{defn}

Note that if we consider a binary relation on a set $S$ as a subset of $S\times S$, refinement is just set inclusion.

The following construction, based on the work of \cite{HillBlumberg,GW,BP,rubin_steiner,rubin_Ninfty,CPCatalan}, links $G$-$N_\infty$ operads and $G$-transfer systems.  Given $\cO\in \NOp^G$, define $\cR_\cO$ by the rule
\[
  K\;\cR_\cO\; H \iff K\le H\text{ and }\cO([H:K])^{\Gamma(H/K)}\ne \varnothing
\]
where $\Gamma(H/K)$ is the graph of some permutation representation $H\to \mathfrak S_{[H:K]}$ of $H/K$.

\begin{thm}
The assignment
\[
\begin{aligned}
  \NOp^G&\longrightarrow \trans(G)\\
  \cO&\longmapsto \cR_\cO
\end{aligned}
\]
induces an equivalence
\[
  \Ho(\NOp^G)\simeq \trans(G)
\]
(considering the poset $\trans(G)$ as a category).
\end{thm}

We conclude this section by recalling a few more facts about transfer systems.

\begin{prop}\label{prop:bounded lattice}
The poset $(\trans(G),\leq)$ is a bounded lattice. 
\end{prop}

\begin{proof}
The least element in $\trans(G)$ is given by the equality relation in $\sub(G)$, while the greatest element is given by the inclusion relation, showing $\trans(G)$ is bounded. The intersection of two transfer systems is a transfer system, thus giving the meet. By \cref{rem:meet-lattice}, we get the desired result.\end{proof}

The join of two transfer systems can be explicitly described as the transfer system generated by their union \cite[Theorem A.2]{rubin_steiner}.

Previous work has revealed the cardinality and structure of transfer systems on the following groups: $C_{p^n}$ \cite{CPCatalan}, $ C_{pq}$ and $C_{pqr}$ \cite{BBPR}, $C_p \times C_p$, $Q_8$,  $S_3$, $D_{2p}$ \cite{rubin_steiner}. 
We expand on the case of $C_{p^n}$, for which the collection of transfer systems is enumerated by the Catalan numbers.

\begin{prop}[{\cite[Theorems 1 and 2]{CPCatalan}}]
\[| \transfers(C_{p^n}) | = \Cat(n+1),\]
where $\Cat(n+1)$ is the $(n+1)$th Catalan number. 
Moreover, the lattice structure on $\transfers(C_{p^n})$ corresponds to the Tamari lattice.
\end{prop}

Here the Tamari lattice is the poset of binary trees with $n+1$ leaves ordered by tree rotation, first explored by D.~Tamari \cite{Tamari}. It forms the 1-dimensional skeleton of $K_{n+2}$, the $n$-dimensional associahedron \cite{Stasheff}.

The original enumeration of $\transfers(C_{p^n})$ in \cite{CPCatalan} proceeds by checking that $|\transfers(C_{p^n})|$ satisfies the recurrence formula for the Catalan numbers. In \cref{sec:nc}, we present an alternate proof based on noncrossing partitions.

\section{A categorical approach to transfer systems}\label{section: Category Stuff}

In this section we define transfer systems for arbitrary finite posets, generalizing the definition for Dedekind groups. We show that for a finite lattice $\cP$, there is a one-to-one correspondence between transfer systems on $\cP$ and weak factorization systems on $\cP$. Using this, we prove that $\trans(\cP)$ is self-dual whenever $\cP$ is a finite self-dual lattice. We compare this result with the involution defined for $G=C_{p_1\dots p_n}$ in \cite{BBPR}. 

\subsection{Transfer systems on posets}\label{trans sys on poset}
We now generalize the definition of a transfer system to an arbitrary poset and characterize them in categorical terms.

\begin{defn}
Let $\cP=(\cP,\leq)$ be a poset. A \emph{transfer system} on $\cP$ consists of a partial order $\rR$ on $\cP$ that refines $\leq$ and such that for all $x,y,z\in \cP$, if $x \rR y$, $z\leq y$, and $x \wedge z$ exists, then $(x \wedge z) \rR z$.
\end{defn}

As noted above, for a Dedekind group $G$, a $G$-transfer system is the same as a transfer system on $\sub(G)$. 

\begin{prop}\label{prop:transfer-pullbacks}
Let $\cP$ be a poset considered as a category. A collection of morphisms $\cR$ is a transfer system if and only if $\cR$ is a subcategory that contains all objects and is closed under pullbacks.
\end{prop}

\begin{proof}
Reflexivity translates to containing the identity morphism for all objects, while transitivity translates to being closed under composition. Note that for a diagram 
\begin{center}
\begin{tikzcd}
{}                   & x \arrow[d] \\
z \arrow[r] & y               
\end{tikzcd}
\end{center}
in $\cP$, the pullback, if it exists, is given by $x\wedge z$. Thus, being closed under restriction translates precisely to being closed under pullbacks.
\end{proof}

\subsection{Weak factorization systems}\label{sec:wfs}

We now explore a surprising connection between transfer systems and weak factorization systems. A standard reference for weak factorization systems is \cite[\S14.1]{MayPonto}. Here we only recall the definitions and basic properties necessary to make the connection with transfer systems.

\begin{defn}\label{defn: lift property}
Let $\cC$ be a category and let $i\colon A \to B$ and $p\colon X \to Y$ be morphisms in $\cC$. If for every $f$ and $g$ that make the square
\begin{center}
\begin{tikzcd}
A \arrow[d, "i"'] \arrow[r, "g"]                       & X \arrow[d, "p"] \\
B \arrow[r, "f"'] \arrow[ru, "\exists\lambda", dotted] & Y               
\end{tikzcd}
\end{center}
commute, there exists a lift $\lambda$ such that the two triangles above commute, we say $i$ has the \emph{left lifting property} with respect to $p$, or equivalently, $p$ has the \emph{right lifting property} with respect to $i$.

\end{defn}

\begin{defn}
Let $\cM$ and $\cN$ be a classes of morphisms in $\cC$. We define
\[
{}^\boxslash \cM = \{i \mid i \text{ has the left lifting property with respect to all } p \in \cM \}
\]
and 
\[
\cM^\boxslash = \{p \mid \text{ has the right lifting property with respect to all } i \in \cM \}.
\]
Note that $\cM \subset {}^\boxslash \cN$ if and only if $\cN \subset \cM^\boxslash$; we write $\cM \boxslash \cN$ when this holds.
\end{defn}

\begin{defn}\label{defn: WFS}
A \emph{weak factorization system} in a category $\cC$ consists of a pair $(\cL,\cR)$ of classes of morphisms in $\cC$ such that 
\begin{enumerate}
    \item every morphism $f$ in $\cC$ can be factored as $f = pi$ with $i \in L$ and $p \in R$, and
    \item $\cL = {}^\boxslash \cR$ and $\cR = \cL^\boxslash.$
\end{enumerate}
\end{defn}

The collection of weak factorization systems on a category $\cC$ forms a poset under inclusion of the right set $\cR$, or equivalently, under reverse inclusion of the left set $\cL$.

 \begin{rem}
 Model categories are closely related to weak factorization systems. Indeed, as proved in \cite{JT}, a model category can be succinctly described as a bicomplete category $\cC$ together with three classes of morphisms, called cofibrations, fibrations and weak equivalences, such that
 \begin{itemize}
     \item (cofibrations, fibrations $\cap$ weak equivalences) and (cofibrations $\cap$ weak equivalences, fibrations) are weak factorization systems;
     \item weak equivalences satisfy the 2-out-of-3 property.
 \end{itemize}
 \end{rem} 
 
\begin{rem}\label{rem:WFS-dual}
 Recall that if $\cC$ is a category, the opposite category $\cC^\op$ is constructed by reversing the direction of morphisms in $\cC$. Note that $(\cL,\cR)$ is a weak factorization system on $\cC$ if and only if $(\cR^\op,\cL^\op)$ is a weak factorization system on $\cC^\op$.
\end{rem}

In order to relate weak factorization systems with transfer systems, we will need a helpful property of weak factorization systems as well as a re-characterization of weak factorization systems.  Proofs of these results can be found, for example, in \cite[\S 14.1]{MayPonto}.

\begin{prop}\label{prop: properties of WFS}
Let $(\cL,\cR)$ be a weak factorization system on a category $\cC$. Then $\cR$ contains all isomorphisms in $\cC$, and is closed under composition, pullbacks, and retracts, and dually for $\cL$.
\end{prop}

\begin{rem}\label{rem:WFS on poset}
In a poset $\cP$, the only isomorphisms are the identity maps, and the only retract of a morphism is itself. Thus, \cref{prop:transfer-pullbacks,prop: properties of WFS} imply that if $(\cL,\cR)$ is a weak factorization system in a lattice $\cP$, then $\cR$ is a transfer system on $\cP$.
\end{rem}

\begin{prop}\label{prop: new WFS}
Let $\cL$ and $\cR$ be a pair of classes of morphisms in $\cC$. Then $(\cL,\cR)$ is a weak factorization system on $\cC$ if and only if 
\begin{enumerate}
    \item every morphism $f$ in $\cC$ can be factored as $f = pi$ with $i \in \cL$ and $p \in \cR$, 
    \item $\cL\boxslash \cR$, and
    \item $\cL$ and $\cR$ are closed under retracts.
\end{enumerate}
\end{prop}


We use this result to construct weak factorization systems from transfer systems on a poset.


\begin{prop}\label{prop: TS to WFS}
Let $\cP$ be a finite lattice and let $\cR$ be a transfer system on $\cP$. Then there is a unique weak factorization system $(\cL,\cR)$ on $\cP$.
\end{prop}

\begin{proof}
For a weak factorization system we need $\cL = {}^\boxslash \cR,$ so $\cL$ is uniquely determined by $\cR$. Since $\cP$ is a poset, the only retract of a morphism $x \to y$ is itself. Thus, by \cref{prop: new WFS}, it suffices to show that every morphism can be factored, since $\cL = {}^\boxslash \cR$ implies $\cL \boxslash \cR$.

Let $x \to y$ be a morphism in $\cP$. If $x \to y$ is in $\cL$ we are done because we can factor with the identity. Suppose then that $x \to y$ is not in $\cL$. This means that there exists a commutative diagram 
\begin{center}
    \begin{tikzcd}
x \arrow[d] \arrow[r] & z \arrow[d] \\
y \arrow[r]           & w          
    \end{tikzcd}
\end{center}
with $z\to w \in \cR$ that does not admit a lift, meaning that $y\not\leq z$, and hence $y\wedge z < y$.

Then, since $y \leq w$ and $\cR$ is closed under restriction, we get that $y \wedge z \to y$ is in $\cR$. If $x \to y \wedge z$ is in $\cL$, we found a factorization of $x\to y$. If not, we can repeat this step, and since $\cP$ is finite, this process must terminate eventually.
\end{proof}

\begin{rem}
The proof makes it clear that one may reformulate \cref{prop: TS to WFS} so that it applies to infinite lattices if we add the condition that $\cR$ is closed under transfinite composition.
\end{rem}
 
\cref{rem:WFS on poset} and \cref{prop: TS to WFS} combine to give the following result.

\begin{thm}\label{thm: 1-1 WFS and TS}
Let $\cP$ be a finite lattice. Then
\[\cR \longleftrightarrow ({}^\boxslash \cR,\cR)\]
gives an isomorphism between the poset of transfer systems on $\cP$ and the poset of weak factorization systems on $\cP$.\hfill\qedsymbol
\end{thm}

We conclude this section with an explicit description of the collection ${}^\boxslash \cR$ for a transfer system $\cR$ on $\cP$.

\begin{defn}\label{defn:downward}
Let $\cP$ be a poset. For a transfer system $\cR \in \trans(\cP)$, define the \emph{downward extension} of $\cR$ to be
\[
\DE(\cR) = \{z\to y \mid \text{there exists } x \in \cP \text{ such that } z\leq x < y \text{ and } x \to y \in \cR\}.
\] 
\end{defn}

\begin{figure}
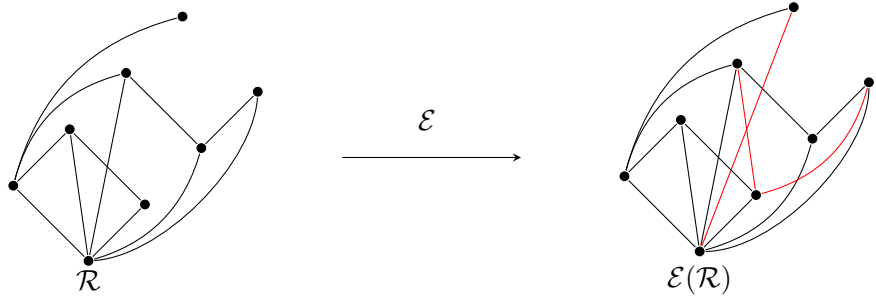

    \centering
    \ctikzfig{transfexample2}
    \caption{A transfer system $\cR$ and the corresponding $\DE(\cR)$ with added edges in red.}
    \label{fig:downward}
\end{figure}

\begin{prop}\label{prop:WFS-downward}
Let $\cR$ be a transfer system on a finite lattice $\cP$. Then 
\[{}^\boxslash \cR=\DE(\cR)^c,\]
where $(-)^c$ denotes the complement of the collection of morphisms. 
\end{prop}

\begin{proof}
We will begin by showing that $\DE(\cR)^c \subseteq {}^\boxslash \cR.$ Let $a \to b \in \DE(R)^c$. Then we need to show that given a commutative diagram
\begin{figure}[H]
    \centering
    \begin{tikzcd}
    a \arrow[r] \arrow[d]          & x \arrow[d] \\
    b \arrow[r] \arrow[ru, dotted] & y          
    \end{tikzcd}
\end{figure}
with $x \to y \in \cR$, there exists a lift. Since we are working with a category coming from a poset, to satisfy the lifting property it will suffice to show that $b \leq x$ in $\cP$.  Note that $b\leq y$, so by restriction we have that $x \wedge b \to b \in \cR$. Moreover, $a \leq x \wedge b$ as $x \wedge b$ is the pullback of $b\to y \leftarrow x$. Thus, the assumption that $a\to b \notin \DE(\cR)$ implies that $x\wedge b = b$, or equivalently, that $b \leq x$ as desired.

For the other inclusion, we will proceed by proving the contrapositive. To that end, suppose that $a \to b \in \DE(\cR)$. Then there exists $x\in \cP$ such that $x\neq b$, $a\leq x$ and $x\to b \in \cR$. Consider the commutative diagram
\begin{figure}[H]
    \centering
\begin{tikzcd}
a \arrow[r] \arrow[d]          & x \arrow[d] \\
b \arrow[r] \arrow[ru, dotted] & b          
\end{tikzcd}
\end{figure}
in $\cP$. Given that $x<b$, there exists no lift, and hence $a\to b$ is not in ${}^\boxslash \cR$, as wanted.
\end{proof}

\subsection{Self-duality}\label{sec:self-dual}
In this section we prove the main result of this paper, namely, that if $\cP$ is self-dual, so is its lattice of transfer systems $\trans(\cP)$.

\begin{defn}\label{defn:self-dual} 
 Let $\cP$ be a poset. We say $\cP$ is \emph{self-dual} if there exists a bijection
 \[ \nabla \colon \cP \to \cP\] 
 such that $x\leq y$ if and only if $y^\nabla \leq x^\nabla$. We call $\nabla$ a \emph{duality} for $\cP$, and write $x^\nabla$ instead of $\nabla(x)$.
\end{defn}

 Note that we can consider $\nabla$ as an isomorphism of categories $\cP^\op \to \cP$. There is no condition on $\nabla$ being an involution,\footnote{Moreover, there are examples of self-dual posets for which there is no order-reversing involution, see \cite[Chapter 3, Exercise 3]{Stanley}.} although it will be in the examples of interest to us.

\begin{ex}\label{ex:grid}
For a natural number $n$, the poset $[n]$ is self-dual by mapping $i$ to $n-i$. This duality extends to the product poset $[n_1]\times \dots \times [n_k]$. 
\end{ex}

\begin{ex}\label{ex:Boolean}
 Recall the Boolean lattice $\cB_n$ of \cref{ex:posets} \eqref{Boolean}. The function that sends a subset to its complement is an order-reversing involution, and hence a duality for $\cB_n$. Note that under the isomorphism $\cB_n \cong [1]^n$, this duality matches with the one in \cref{ex:grid}.
\end{ex}

\begin{ex}\label{ex:divisor}
 The poset $\cD_n$ of positive divisors of $n$ is self-dual by mapping $k$ to $n/k$. Recall from \cref{ex:posets} \eqref{divisors} that if $n=p_1^{a_1}\dots p_j^{a_k}$ is the prime decomposition of $n$, then $\cD_n$ is isomorphic to $[a_1]\times \dots \times [a_k]$. Under this isomorphism, the duality of $\cD_n$ coincides with the one of \cref{ex:grid}.
\end{ex}

\begin{ex}\label{ex:abeliansub}
As noted in \cite[Theorem 8.1.4]{schmidt}, the lattice of subgroups $\sub(G)$ is self-dual for every finite Abelian group $G$. The bijection $\nabla$ is constructed using a (non-canonical) isomorphism between $G$ and $G^*=\Hom(G,\Co^\times)$. In the case that $G=C_n$, there is an explicit order-reversing involution given by $C_k \mapsto C_{n/k}$, which coincides with the one in \cref{ex:divisor} via the identification of $\sub(C_n)$ with $\cD_n$.
\end{ex}

\begin{thm}\label{thm:cat-involution}
 If $(\cP,\nabla)$ is a self-dual lattice, then $\trans(\cP)$ is self-dual, with duality
 \[\phi \colon \trans(\cP) \to \trans(\cP)\]
 given by 
 \[\phi(\cR)=(({}^\boxslash\cR)^\op)^\nabla.\]
 Moreover, if $\nabla$ is an involution so is $\phi$.
\end{thm}

\begin{proof}
Given a transfer system $\cR$ on $\cP$, we consider the weak factorization system $({}^\boxslash \cR, \cR)$ on $\cP$. By \cref{rem:WFS-dual},  $(\cR^\op, ({}^\boxslash\cR)^\op)$ is a weak factorization system on $\cP^\op$. Since $\nabla$ is an isomorphism, it takes this to a weak factorization system $((\cR^\op)^\nabla, (({}^\boxslash\cR)^\op)^\nabla)$ on $\cP$. Thus, by \cref{rem:WFS on poset} the collection $(({}^\boxslash\cR)^\op)^\nabla$ is a transfer system on $\cP$.

If $\cR'$ is another transfer system on $\cP$, we have that $\cR \subseteq \cR'$ if and only if ${}^\boxslash \cR' \subseteq {}^\boxslash \cR$. This shows that $\cR \subseteq \cR'$ if and only if $\phi(\cR')\subseteq \phi(\cR)$, since $(-)^\op$ and $(-)^\nabla$ preserve and reflect the containment relation, respectively. 

To prove that $\phi$ is a bijection, let $\Delta\colon \cP^\op\to \cP$ denote the opposite of the inverse of $\nabla$, that is, $\Delta ^{-1}=\nabla^\op$. Using that $({}^\boxslash \cM)^\op = (\cM^\op)^\boxslash$ and that $\nabla$ is an isomorphism, we get that $\phi$ can also be expressed as
\[\phi(\cR)=((\cR^{\nabla^\op})^\op)^\boxslash.\]
Thus, the fact that $({}^\boxslash \cR)^\boxslash = \cR$ (see \cref{prop: TS to WFS}) implies that $\phi$ is a bijection with inverse
\[\phi^{-1}(\cR)=(({}^\boxslash \cR)^\op)^\Delta.\]
In particular, if $\nabla$ is an involution, so is $\phi$.

\end{proof}

As a consequence of \cref{prop:WFS-downward}, we obtain an explicit description of the involution $\phi$. See \cref{duality} for an example of this result.

\begin{cor}\label{cor:phi-explicit}
Let $(\cP,\nabla)$ be a self-dual lattice. Then the involution $\phi$ satisfies that
\[\phi (\cR)=((\DE(\cR)^\op)^\nabla)^c.\]
\hfill\qedsymbol
\end{cor}


\subsection{Slats --- numerical symmetry of the duality}\label{sec:slats}
We expect the existence of the duality $\phi$ to aid in proving enumeration results, as was done in \cite{BBPR} to count the number of transfer systems for $[1]\times [1] \times [1]$. In fact, we first suspected the existence of an involution when attempting to enumerate the transfer systems for $[n]\times [1]$ and noticed there was a symmetry in the results when we restricted to counting transfer systems with a given number of ``slats'', as we now explain. Throughout this section we fix $n\geq 1$.

\begin{defn}\label{defn:slats}
Let $0\leq k \leq n$. The $k$th \emph{slat} in the poset $[n]\times[1]$ is $(k,0) \le (k,1)$. If $\cR$ is a transfer system on $[n]\times [1]$, by the restriction property, if the $k$th slat is in $\cR$, so is the $i$th slat for all $0\leq i<k$. We say that the \emph{top slat} in a transfer system $\rR$ on $[n]\times [1]$ is the slat with the largest $k$ such that $(k,0)\rR (k,1)$. For $0\leq k \leq n$, we let $S_k$  denote the set of all transfer systems for which the $k$th slat is the top slat. We let $S_{-1}$ be the set of transfer systems with no slats. 
\end{defn}

Note that we can alternatively describe $S_k$ as the collection of transfer systems with $k+1$ slats.

Recall the involution $\nabla$ on $[n]\times [1]$ of \cref{ex:grid}. It sends $(i,j)$ to $(n-i,1-j)$.

\begin{prop}
 The involution $\phi$ of \cref{thm:cat-involution} exchanges the sets $S_k$ and $S_{n-k-1}$.
\end{prop}

\begin{proof}
 We use the explicit description of $\phi$ from \cref{cor:phi-explicit}. Since slats are covering relations in $[n]\times[1]$, a given slat is in $\cR$ if and only if it is in $\DE(\cR)$. Note furthermore that $\nabla$ sends the $k$th slat to the $(n-k)$th slat. It follows that $\cR$ contains the $k$th slat if and only if $\phi(\cR)$ does not contain the $(n-k)$th slat, thus proving the result.
\end{proof}

\begin{center}
\begin{figure}
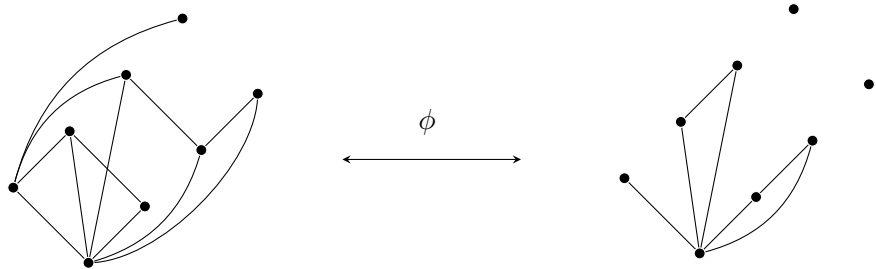

    \ctikzfig{41duality}
    \caption{The involution $\phi$ acting on a $[3]\times [1]$ transfer system. Using \cref{cor:phi-explicit}, the transfer system on the right is obtained from that on the left by applying $\DE$ (as was demonstrated in \cref{fig:downward}), rotating the result by $180^\circ$ (this has the effect of taking the opposite poset), and taking the complement.     Note that the involution exchanges the 3-slat transfer system on the left with the 1-slat transfer system on the right.}\label{duality}
\end{figure}
\end{center} 

\subsection{Cyclic groups of squarefree order}\label{sec:BBPR}

In \cite{BBPR}, Balchin and collaborators define an order-reversing involution on the lattice of transfer systems for the cyclic group of order $p_1\dots p_n$, where $p_1,\dots,p_n$ are distinct primes. In this section we prove that their involution coincides with the one defined above.

\begin{rem}\label{rem:Boolean}
 Let $G=C_{p_1\dots p_n}$. As mentioned in \cref{ex:posets}, the lattice $\sub(G)$ is isomorphic to the Boolean lattice $\cB_n$, and the involutions of \cref{ex:Boolean,ex:abeliansub} coincide via this isomorphism.
  
 The Hasse diagram of $\cB_n$ consists of the edges of an $n$-dimensional cube, and thus we can consider its $2n$ facets. Borrowing notation from \cite{BBPR}, for all $i=1,\dots n$, we denote by $B_i$ and $T_i$ the bottom and top facets, respectively. These correspond to the facets bounded by the vertices $a\in \{0,1\}^n$ with $a_i=0$ for the bottom and $a_i=1$ for the top. Given a transfer system $\cR$ on $\cB_n$, we can restrict it to a facet to obtain a transfer system therein. See \cref{ex:Phi} for an explicit example of the restriction operation.
\end{rem}

We now recall the involution $\Phi_n \colon \trans(\cB_n) \to \trans(\cB_n)$ of \cite[\S4.2]{BBPR}. The presentation here contains a minor clarification communicated to us by the authors.

\begin{const}\label{cons:phin}
For $n\geq 1$, the involution $\Phi_{n}$ on $\trans(\cB_n)$ is defined inductively as follows: 

\begin{itemize}
\item If $n=1$, $\Phi_1$ exchanges the trivial transfer system with the full transfer system.
\item Suppose $\Phi_n$ is defined for some $n\geq 1$, and let $\cR\in \trans(\cB_{n+1})$. Then $\Phi_{n+1}(\cR)$ is obtained by applying $\Phi_{n}$ to $\cR$ restricted to each facet, and placing the result in the opposite facet. Lastly, we add the long diagonal edge $\vec{0}\to \vec{1}$ if $\cR$ did not contain any nontrivial edges with target $\vec{1}$.\\
\end{itemize}
\end{const}

Balchin \emph{et al.} prove that this function is well defined and is indeed an order-reversing involution on $\trans(\cB_n)$ (see \cite[Theorem 4.5, Proposition 4.6]{BBPR}).

\begin{ex}\label{ex:Phi}
We will show how to compute $\Phi_3(\cR)$ where $\cR$ is the following transfer system on $\cB_3$.
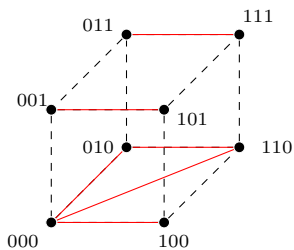
\begin{center}
    \ctikzfig{CubeDiagram}
\end{center} 

We first perform $\Phi_2$ to $\cR$ restricted to each of the six facets and place the result on the opposite facet. For example, to obtain the restriction of $\Phi_3(\cR)$ to $T_3$, we perform $\Phi_2$ restricted to $B_3$.
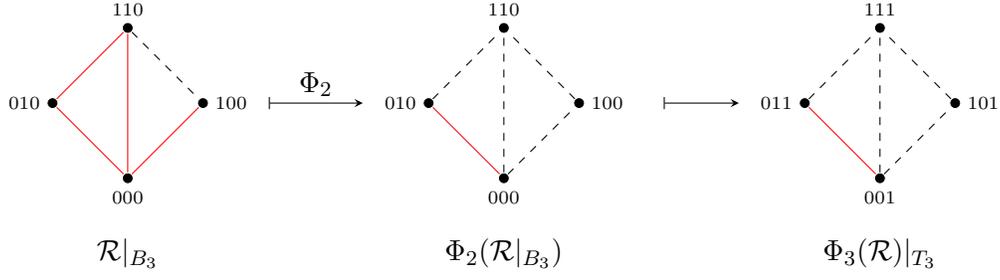
\begin{center}
    \ctikzfig{FacetDiagramMap1}
\end{center}
After this is done to each facet, we assemble the results. At the end, we decide whether or not to include the long diagonal. In this case, since $\cR$ contains the edge from $(0,1,1)$ to $(1,1,1)$, we do not include the long diagonal. The diagram for $\Phi_3(\cR)$ is shown below.
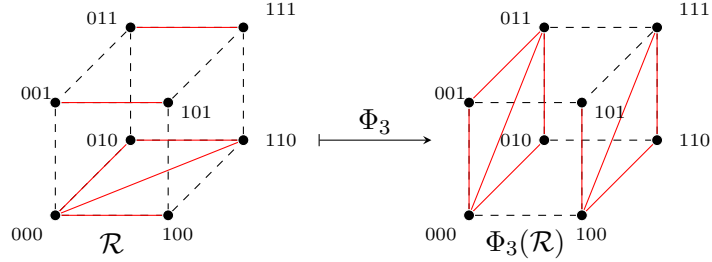
\begin{center}
    \ctikzfig{CubeDiagramMap}
\end{center}
\end{ex}

\begin{thm}\label{thm:compare BBPR}
 Consider the Boolean lattice $\cB_n$ with duality $\nabla$ given by swapping $0$ and $1$. Then the involution $\phi$ of \cref{thm:cat-involution} is equal to $\Phi_n$.
\end{thm}

\begin{proof}
We proceed by induction, using the explicit description of $\phi$ from \cref{cor:phi-explicit}, which uses the downward extension of \cref{defn:downward}. A quick calculation shows that the result holds when $n=1$. 

Assume the result holds for the $n$-dimensional Boolean lattice, and let $\cR$ be a transfer system on $\cB_{n+1}$.  First notice that the long diagonal $\vec{0}\to \vec{1}$ is in $\phi(\cR)$ if and only if it is not in $\DE(\cR)$. By the definition of $\DE(\cR)$, this happens exactly when $\cR$ contains no nontrivial edges with target $\vec{1}$, as needed. Thus, it remains to check that $\phi$ and $\Phi_n$ coincide on the boundary of the cube.

Let $F$ be a facet of $\cB_{n+1}$, and let $\tilde{F}$ denote its opposite facet, i.e., if $F=T_i$, then $\tilde{F}=B_i$, and vice versa. Recall from \cref{cons:phin} that the restriction of $\Phi_{n+1}(\cR)$ to $F$ is obtained by taking $\Phi_n(\cR|_{\tilde{F}})$.

Thus, by the inductive hypothesis and \cref{cor:phi-explicit}, we have that, via the canonical identification of $F$ and $\tilde{F}$ with $\cB_n$ obtained by dropping the $i$th coordinate,
\[\Phi_{n+1}(\cR)|_F=\Phi_n(\cR|_{\tilde{F}})=\phi(\cR|_{\tilde{F}})=((\DE_{\tilde{F}}(\cR |_{\tilde{F}})^\op)^{\nabla_{\tilde{F}}})^{c_{\tilde{F}}}.\]
Here $\DE_{\tilde{F}}$ and $\nabla_{\tilde{F}}$ denote the downward extension and the duality within $\tilde{F}$, and $c_{\tilde{F}}$ takes the complement within $\tilde{F}$ as well.  Hence it suffices to prove that for each facet $F$
\begin{equation}\label{eq:de_res}
(((\DE(\cR )^\op)^\nabla)^c)|_F=((\DE_{\tilde{F}}(\cR |_{\tilde{F}})^\op)^{\nabla_{\tilde{F}}})^{c_{\tilde{F}}},
\end{equation}
again, modulo the identification of the facets with $\cB_n$.

The map $\nabla$ swaps 0 and 1 in all coordinates, thus, its restriction to a facet $F$ can be obtained by restricting to $\tilde{F}$ (i.e., swapping the $i$th coordinate if $F$ is $T_i$ or $B_i$), and performing $\nabla_{\tilde{F}}$ on it (i.e., swapping all coordinates but the $i$th one). Thus \eqref{eq:de_res} reduces to
\begin{equation}\label{eq:de_res2}
(\DE(\cR )^c)|_{\tilde{F}}
=\DE_{\tilde{F}}(\cR |_{\tilde{F}})^{c_{\tilde{F}}}.
\end{equation}

Since $\tilde{F}$ is an interval in $\cB_{n+1}$ (it contains all the vertices between its bottom and top vertices), we have that
\[\DE_{\tilde{F}}(\cR | _{\tilde{F}})=\DE(\cR)|_{\tilde{F}}.\] 
Using that the restriction of the complement is the complement of the restriction, we get \eqref{eq:de_res2}, and hence the inductive step.




\end{proof}

\section{Transfer systems and noncrossing partitions}\label{sec:nc}

While it is already known that $C_{p^n}$-transfer systems can be counted with Catalan numbers, we re-prove this result by presenting a natural bijection between noncrossing partitions of the set $\{0,\dots,n\}$ and $C_{p^n}$-transfer systems. Additionally, we will use this bijection and the structure of noncrossing partitions in order to relate Narayana numbers and transfer systems. We first recall the definition of a noncrossing partition.

Recall that $\sub(C_{p^n})$ is isomorphic to the poset $[n]=\{0<1<\dots<n\}$ of \cref{ex:posets}\eqref{linear}. Thus for ease of notation we work with transfer systems on $[n]$.  

\begin{defn}\label{defn:noncrossing}
A partition of the set $\{0,\dots,n\}$ is \emph{noncrossing} if for all $0 \leq a < b < c < d \leq n$ such that $a$ and $c$ are in the same class and $b$ and $d$ are in the same class, then $a$, $b$, $c,$ and $d$ are all in the same class. We let $\NC_{n+1}$ denote the set of all noncrossing partitions of the set $\{0, \dots, n\}$. 
\end{defn}

Observe that the index in $\NC_{n+1}$ corresponds to the cardinality of $[n]$. It is well-known (see \cite[Exercise 6.19(pp)]{Stanley2} and \cite{becker}) that noncrossing partitions are counted by Catalan numbers:
\[
  |\NC_n| = \Cat(n) = \frac{1}{n+1}\binom{2n}{n}.
\]

In order to construct a bijection between noncrossing partitions and transfer systems in $[n]$, we need some key definitions.

\begin{defn}\label{defn:M(X)}
Let $\cR \in \transfers([n])$. We say $i \to j \in \cR$ is a \emph{maximal edge in $\cR$} if $i\neq j$ and there exists no $k\geq j$ such that $i \to k \in \cR$. We denote the set of all maximal edges in $\cR$ by $M(\cR)$.
\end{defn}

\begin{figure}[H]
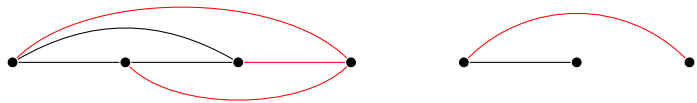

    \ctikzfig{9}
    \caption{Example of maximal edges in a transfer system in the lattice [6]. The maximal edges are marked in red.}
    \label{fig:my_label}
\end{figure}

Note that there is at most one maximal edge starting at a given vertex. We now construct the bijection between  $\transfers([n])$ and $\NC_{n+1}$.

\begin{defn}\label{defn: Rubin closure} 
Let $(\cP,\leq)$ be a poset, and let $\F$ be a binary relation on $\cP$ that refines $\leq$. We denote by $\langle \F \rangle$ the minimal transfer system that contains $\F$ as given in \cite[Construction A.1]{rubin_steiner}. We say $\F$ is a \emph{generating set} for $\langle \F \rangle$. 

For a transfer system $\mathcal{R}$ on $\cP$, we define the \emph{minimal generating number} of $\mathcal{R}$ to be the minimal cardinality of a generating set for $\cR$. Any generating set with minimal cardinality is called a \emph{minimally generating set}.
\end{defn}

The following lemma tells us that the maximal edges of $\cR$ form a minimally generating set for $\cR$.

\begin{lemma}\label{prop: gen by min}
Let $\cR\in \transfers([n])$. Then $\cR=\langle M(\cR) \rangle$, and moreover, $M(\cR)$ is a minimally generating set for $\cR$.
\end{lemma}

\begin{proof}
The equality follows from the fact that any nontrivial edge in $\cR$ is the restriction of an edge in $M(\cR)$.

Consider the set $S$ of vertices in $[n]$ that are the sources of nontrivial edges in $\cR$.
If $i \in S$ then $i \to (i+1) \in \cR$ by restriction. Since edges of this form are covering relations in $[n]$, we know these edges cannot be obtained by closing under transitivity. Thus the only edges that could generate the edge $i\to (i+1)$ are precisely edges of the form $i\to j$ where $j \geq i+1$. From this it is then clear that in order to generate $\cR$, a generating set must have a size of at least $\lvert S \rvert$.  Note that a vertex $i$ is in $S$ if and only if it is the source of an edge in $M(\cR)$. Thus $\lvert M(\cR) \rvert = \lvert S \rvert$, and the claim follows. 
\end{proof}

\begin{defn}\label{defn: G(P)}
Let $\pi$ be a noncrossing partition of $[n]$. Let $J(\pi)$ be the set of nontrivial edges in $[n]$ connecting each element of a block of $\pi$ with the largest element in that block.
\end{defn}

\begin{figure}[H]
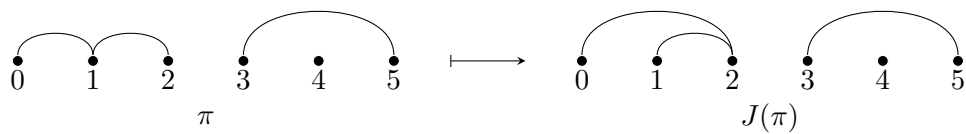

\ctikzfig{GP}
\caption{An example of a noncrossing partition $\pi$ and edge set $J(\pi)$. Here $\pi = \{\{0,1,2\},\{3,5\},\{4\}\}$ with arcs connecting nearest elements of each block; that these arcs do not intersect indicates that $\pi$ is indeed noncrossing.}
\end{figure}

\begin{lemma}\label{lem:gen_by_J}
Let $\pi$ be a noncrossing partition of $[n]$. Then the transfer system $\langle J(\pi)\rangle$ consists of the closure under restriction of $J(\pi)$.
\end{lemma}

\begin{proof}
As noted in \cite[Construction A.1]{rubin_steiner}, $\langle J(\pi)\rangle$ is constructed by first closing under restriction and then under transitivity. Thus, it is enough to prove that the closure under restriction of $J(\pi)$ is already closed under transitivity. Suppose $i\to j$ and $j\to k$ are obtained by restriction from $J(\pi)$. This means that there exist $j'\geq j$ and $k'\geq k$ such that $i$ and $j'$ are in the same class and $j'$ is the largest element therein, and similarly, $j$ and $k'$ are in the same class and $k'$ is the largest element therein. Since $\pi$ is noncrossing, we must have that $k'\leq j'$. Then $i\to k$ is the restriction of $i\to k'$, which is in $J(\pi)$. 
\end{proof}

\begin{thm}\label{thm:nc-bijection}
 Let $\cR \in \transfers([n])$. Let $\psi(\cR)$ be the partition of $[n]$ associated to the equivalence relation generated by $M(\cR)$. Then
 \begin{enumerate}
  \item $\psi(\cR)$ is a noncrossing partition, and
  \item the map $\psi \colon \transfers([n]) \to \NC_{n+1}$ is a bijection with inverse $\chi(\pi)=\langle J(\pi)\rangle$.
 \end{enumerate}
\end{thm}

\begin{proof}
 We denote by $\sim_\cR$ the equivalence relation generated by $M(\cR)$, \emph{i.e.}, $\sim_\cR$ is the intersection of all equivalence relations containing $M(\cR)$ where directed edges in $M(\cR)$ are interpreted as relations.  Using the transitivity of a transfer system, one can check that for $i<j \in [n]$, $i\sim_\cR j$ if and only if either $i\to j \in M(\cR)$ or there exists $k>j$ such that $i \to k \in M(\cR)$ and $j \to k \in M(\cR)$. Note that in either case, there exists $k\geq j$ such that $i\to k \in M(\cR)$ and $j\to k \in \cR$.
  
 To prove that $\psi(\cR)$ is a noncrossing partition, let $a<b<c<d$ in $[n]$ such that $a\sim_\cR c$ and $b \sim_\cR d$. Thus, there exist $e\geq c$ and $f\geq d$ such that $a\to e$ and $b\to f$ are in $M(\cR)$, and $c\to e$ and $d\to f$ are in $\cR$. By restriction, we get that $a\to b$ and $b\to c$ are in $\cR$, and thus $a\to f$ and $b\to e$ are also in $\cR$ by transitivity. Since $a\to e$ and $b\to f$ are maximal, that implies that $e=f$, and hence $a$, $b$, $c$, and $d$ are all in the same equivalence class, as required.
 
We now prove that $\psi$ is a bijection. Note that \cref{lem:gen_by_J} implies that for a noncrossing partition $\pi$,
\[J(\pi)=M(\chi(\pi)).\]
Since the partition generated by $J(\pi)$ is precisely $\pi$, it follows that $\psi(\chi(\pi))=\pi$.

Similarly, the description above of the equivalence relation generated by $M(\cR)$ implies that 
\[M(\cR)=J(\psi(\cR)).\]
That combined with \cref{prop: gen by min} proves that $\chi(\psi (\cR)) = \cR$, thus showing that $\psi$ and $\chi$ are indeed inverses of each other.

 \end{proof}

\begin{rem}
The above theorem recovers \cite[Theorem 1]{CPCatalan} via a direct bijection with noncrossing partitions. The original proof in \cite{CPCatalan} proceeds by establishing the Catalan recurrence relation among the numbers $|\transfers([n])|$.
\end{rem}

With $\psi$ established as a bijection, we can now use $\psi$ to directly enumerate transfer systems minimally generated by edge sets of a certain size. Note that refinement of partitions gives an order relation on the collection of noncrossing partitions.  

\begin{prop}[{\cite{Kreweras}, see also \cite{Narayana}}] 
The poset of $\NC_n$ under refinement forms a lattice called the \emph{Kreweras lattice}. It is graded by the function 
\[\operatorname{rank}(\pi) := n - \operatorname{bk}(\pi),\] 
where $\operatorname{bk}(\pi)$ is the number of blocks of $\pi$.\hfill\qedsymbol 
\end{prop}

\begin{rem}
Beware that the Tamari lattice is a strict extension of the Kreweras lattice.
\end{rem}

\begin{defn}
For $n\ge 1$ and $1\le k\le n$, the $(n,k)$-th \emph{Narayana number} is
\[
  \mathrm{Nar}(n,k) := \frac{1}{n}\binom nk \binom n{k-1}.
\]
\end{defn}

\begin{prop}[{\cite{Kreweras}, see also \cite{Narayana}}]\label{prop: Narayana rank} 
Let $\NC_n(k)$ denote the number of partitions in $\NC_n$ with rank $k$. Then $\NC_n(k)$ is given by the $(n,k)$-th Narayana number, \emph{i.e.},
\[
\NC_n(k) = \frac{1}{n}\binom{n}{k}\binom{n}{k-1}.
\]
\hfill\qedsymbol
\end{prop}

With this rank property established in the context of noncrossing partitions, we can then use $\chi$ to see what this rank looks like for transfer systems. 

\begin{prop}\label{prop: NC maximal edges}
Let $\pi \in \NC_{n+1}$. Then the rank of $\pi$ in the Kreweras lattice is equal to the minimal generating number for $\chi(\pi)$.
\end{prop}

\begin{proof}
By \cref{prop: gen by min}, the minimal generating number for $\chi(\pi)$ is equal to the cardinality of $M(\chi(\pi))$. As noted in the proof of \cref{thm:nc-bijection}, $M(\chi(\pi))=J(\pi)$. Note that $J(\pi)$ consists precisely of the elements of $[n]$ that are not maximal within their block in $\pi$, thus, the cardinality of $J(\pi)$ is $n+1-\operatorname{bk}(\pi)$, as desired.
\end{proof}

Let $\transfers_k([n])$ denote the set of transfer systems on $[n]$ minimally generated by $k$ edges.

\begin{cor}
For all $1\le k\le n+1$,
\[
\lvert\transfers_k([n])\rvert = \NC_{n+1}(k)=\frac{1}{n+1}\binom{n+1}{k}\binom{n+1}{k-1}.
\]
\hfill\qedsymbol
\end{cor}
\begin{proof}
The bijection $\chi$ takes rank $k$ noncrossing partitions of $[n]$ to transfer systems on $[n]$ minimally generated by $k$ edges. Since the former are counted by the indicated Narayana numbers, so are the latter.
\end{proof}


\bibliographystyle{amsalpha}
\bibliography{bib}

\providecommand{\bysame}{\leavevmode\hbox to3em{\hrulefill}\thinspace}
\providecommand{\MR}{\relax\ifhmode\unskip\space\fi MR }
\providecommand{\MRhref}[2]{%
  \href{http://www.ams.org/mathscinet-getitem?mr=#1}{#2}
}
\providecommand{\href}[2]{#2}
\begin{thebibliography}{Tam62}

\bibitem[BBPR]{BBPR}
Scott Balchin, Daniel Bearup, Clelia Pech, and Constanze Roitzheim,
  \emph{Equivariant homotopy commutativity for ${G}={C}_{pqr}$}, arXiv e-prints
  (2020), arXiv:2001.05815. To appear in Tbilisi Mathematical Journal.

\bibitem[BBR]{CPCatalan}
Scott Balchin, David Barnes, and Constanze Roitzheim,
  \emph{${N}_\infty$-operads and associahedra}, arXiv e-prints (2019),
  arXiv:1905.03797v2.

\bibitem[Bec48]{becker}
H.~W. Becker, \emph{Rooks and {R}hymes}, Math. Mag. \textbf{22} (1948), no.~1,
  23--26.

\bibitem[BH15]{HillBlumberg}
Andrew~J. Blumberg and Michael~A. Hill, \emph{Operadic multiplications in
  equivariant spectra, norms, and transfers}, Adv. Math. \textbf{285} (2015),
  658--708.

\bibitem[BP21]{BP}
Peter Bonventre and Lu\'{\i}s~A. Pereira, \emph{Genuine equivariant operads},
  Adv. Math. \textbf{381} (2021), 107502, 133.

\bibitem[GW18]{GW}
Javier~J. Guti\'{e}rrez and David White, \emph{Encoding equivariant
  commutativity via operads}, Algebr. Geom. Topol. \textbf{18} (2018), no.~5,
  2919--2962.

\bibitem[JT07]{JT}
Andr\'{e} Joyal and Myles Tierney, \emph{Quasi-categories vs {S}egal spaces},
  Categories in algebra, geometry and mathematical physics, Contemp. Math.,
  vol. 431, Amer. Math. Soc., Providence, RI, 2007, pp.~277--326.

\bibitem[Kre72]{Kreweras}
G.~Kreweras, \emph{Sur les partitions non crois\'{e}es d'un cycle}, Discrete
  Math. \textbf{1} (1972), no.~4, 333--350.

\bibitem[MP12]{MayPonto}
J.~Peter May and Kate Ponto, \emph{More concise algebraic topology}, Chicago
  Lectures in Mathematics, University of Chicago Press, Chicago, IL, 2012.

\bibitem[Rub]{rubin_Ninfty}
Jonathan Rubin, \emph{Combinatorial {$N_\infty$} operads}, arXiv e-prints
  (2020), arXiv:1705.03585v3.

\bibitem[Rub20]{rubin_steiner}
\bysame, \emph{Detecting {S}teiner and linear isometries operads}, Glasgow
  Mathematical Journal (2020), 1--36.

\bibitem[Sch94]{schmidt}
Roland Schmidt, \emph{Subgroup lattices of groups}, De Gruyter Expositions in
  Mathematics, vol.~14, Walter de Gruyter \& Co., Berlin, 1994.

\bibitem[Sim00]{Narayana}
Rodica Simion, \emph{Noncrossing partitions}, Discrete Math. \textbf{217}
  (2000), no.~1-3, 367--409, Formal power series and algebraic combinatorics
  (Vienna, 1997).

\bibitem[Sta63]{Stasheff}
James~Dillon Stasheff, \emph{Homotopy associativity of {$H$}-spaces. {I}},
  Trans. Amer. Math. Soc. 108 (1963) \textbf{108} (1963), 275--292.

\bibitem[Sta99]{Stanley2}
Richard~P. Stanley, \emph{Enumerative combinatorics. {V}ol. 2}, Cambridge
  Studies in Advanced Mathematics, vol.~62, Cambridge University Press,
  Cambridge, 1999, With a foreword by Gian-Carlo Rota and appendix 1 by Sergey
  Fomin.

\bibitem[Sta12]{Stanley}
\bysame, \emph{Enumerative combinatorics. {V}olume 1}, second ed., Cambridge
  Studies in Advanced Mathematics, vol.~49, Cambridge University Press,
  Cambridge, 2012.

\bibitem[Tam62]{Tamari}
Dov Tamari, \emph{The algebra of bracketings and their enumeration}, Nieuw
  Arch. Wisk. (3) \textbf{10} (1962), 131--146.

\end{thebibliography}

\end{document}